\pdfoutput=1
\documentclass[11pt,a4paper]{amsart}
\usepackage[margin=1in]{geometry}
\usepackage[utf8]{inputenc}
\usepackage[T1]{fontenc}
\usepackage{textcase}
\usepackage[dvipsnames]{xcolor}
\usepackage{microtype}

\usepackage{amsmath}
\usepackage{amssymb}
\usepackage{eucal}
\usepackage{mathrsfs}
\usepackage{tikz-cd}

\usepackage{enumitem}
\usepackage[pdfusetitle,colorlinks]{hyperref}
\hypersetup{bookmarksdepth=2,pdfencoding=unicode,allcolors=MidnightBlue}

\usepackage{zref-clever}
\zcsetup{abbrev=false,cap=true,nameinlink=false,sort=false,lang=english}
\newcommand{\cref}[1]{\zcref{#1}}
\newcommand{\Cref}[1]{\zcref[S]{#1}}
\zcsetup{pairsep={ and~},lastsep={, and~}}
\zcRefTypeSetup{equation}{Name-sg=,Name-pl=,refbounds={(,,,)}}
\AddToHook{env/equation/begin}{\zcsetup{countertype={equation=equation}}}
\AddToHook{env/align/begin}{\zcsetup{countertype={equation=equation}}}
\zcRefTypeSetup{item}{Name-sg=,Name-pl=,refbounds={(,,,)}}
\newlist{conenum}{enumerate}{1}
\setlist[conenum,1]{label=(\roman*),ref=\roman*}
\zcRefTypeSetup{conenumi}{Name-sg=,Name-pl=,refbounds={(,,,)}}

\NewDocumentCommand{\newzctheorem}{momo}{\IfValueTF{#4}
  {\newtheorem{#1}{#3}[#4]}
  {\IfValueTF{#2}
    {\AddToHook{env/#1/begin}{\zcsetup{countertype={#2=#1}}}\newtheorem{#1}[#2]{#3}}
    {\newtheorem{#1}{#3}}}}
\numberwithin{equation}{section}
\theoremstyle{plain}
\newzctheorem{Theorem}{Theorem}

\zcRefTypeSetup{Theorem}{Name-sg=Theorem,Name-pl=Theorems}
\newzctheorem{theorem}[equation]{Theorem}
\newzctheorem{proposition}[equation]{Proposition}
\newzctheorem{lemma}[equation]{Lemma}
\newzctheorem{corollary}[equation]{Corollary}

\theoremstyle{definition}
\newzctheorem{definition}[equation]{Definition}
\newzctheorem{example}[equation]{Example}

\theoremstyle{remark}
\newzctheorem{remark}[equation]{Remark}

\let\oldAA\AA\let\AA\relax
\let\oldSS\SS\let\SS\relax

\newcommand{\ZZ}{\mathbf{Z}}
\newcommand{\FF}{\mathbf{F}}
\newcommand{\QQ}{\mathbf{Q}}
\newcommand{\CC}{\mathbf{C}}
\newcommand{\SS}{\mathbf{S}}

\newcommand{\AA}{\mathbb{A}}

\newcommand{\et}{\textnormal{ét}}
\newcommand{\lis}{\textnormal{lis}}
\newcommand{\mot}{\textnormal{mot}}
\newcommand{\Bet}{\textnormal{Bet}}
\newcommand{\cdh}{\textnormal{cdh}}

\newcommand{\Spec}{\operatorname{Spec}}
\newcommand{\op}{\operatorname{op}}

\newcommand{\KH}{\operatorname{KH}}
\newcommand{\ku}{\operatorname{ku}}
\newcommand{\gr}{\operatorname{gr}}
\newcommand{\fil}{\operatorname{fil}}
\newcommand{\CH}{\operatorname{CH}}

\newcommand{\D}{\operatorname{D}}

\newcommand{\X}{\mathord{-}}

\newcommand{\Cls}[1]{\mathscr{#1}}
\newcommand{\Cat}[1]{\mathsf{#1}}

\title{On motivic cohomology of commutative C*-algebras}
\author{Ko Aoki}
\address{Department of Mathematical Sciences,
  University of Copenhagen, Denmark
}
\email{aoki@math.ku.dk}
\date{\today}

\begin{document}

\begin{abstract}
  Consider the ring
  of complex-valued continuous functions
  on a compact Hausdorff space~\(X\).
  Cortiñas–Thom computed its nonpositive algebraic \(K\)-theory,
  confirming Rosenberg’s expectation.
  We give a motivic refinement of their result:
  For \(i\geq2j\geq0\),
  the \(i\)th integral motivic cohomology group
  of weight~\(j\)
  is naturally isomorphic
  to the \(i\)th Betti cohomology group of~\(X\).
\end{abstract}

\maketitle

\section{Introduction}\label{s:intro}

A \emph{compactum} means a compact Hausdorff space.
For a compactum~\(X\),
we write
\(
  \Cls{C}(X)
\)
for the ring
of complex-valued continuous functions on~\(X\).
These rings are precisely
the unital commutative complex C*-algebras.
For a local compactum, Betti cohomology means sheaf (aka Čech) cohomology.

\subsection{Results}\label{ss:results}

The following appeared in~\cite[Corollary~6.9]{CortinasThom12},
which confirmed Rosenberg’s expectation:

\begin{theorem}[Cortiñas–Thom]\label{xhar7u}
  For every compactum~\(X\),
  there is a natural isomorphism
  \begin{equation*}
    K_{*}(\Cls{C}(X))\simeq\ku^{-{*}}(X)
  \end{equation*}
  for \({*}\leq0\),
  where \(K\)
  and~\(\ku\)
  denote (nonconnective) algebraic \(K\)-theory
  and connective complex topological \(K\)-theory,
  respectively.
\end{theorem}

An algebraic proof,
valid more generally for functions with values
in a local division ring, was given in~\cite{k-ros-1}.
See~\cite[Section~1.1]{k-ros-1}
for a comparison of the methods.

Elmanto–Morrow~\cite{ElmantoMorrow}
constructed a \emph{motivic filtration} on algebraic \(K\)-theory
for equicharacteristic rings.
In this paper, we work in characteristic zero.
For a \(\QQ\)-algebra~\(A\),
the weight-\(j\) graded piece of this filtration is
the shifted \emph{motivic cohomology}
\begin{equation*}
  \gr^j_{\mot}K(A)
  \simeq \ZZ(j)^{\mot}(A)[2j].
\end{equation*}
Our result computes the motivic cohomology of~\(\Cls{C}(X)\)
in the range corresponding to~\(K_{*}\) for~\({*}\leq0\):

\begin{Theorem}\label{main}
  Let~\(X\) be a compactum
  and \(j\geq0\) and~\(i\) integers.
  Betti realization induces a natural map
  \begin{equation*}
    H^i_{\mot}(\Cls{C}(X);\ZZ(j))
    \to
    H^i_{\Bet}(X;\ZZ(j)),
  \end{equation*}
  which is bijective for~\(i\geq2j\)
  and surjective for~\(i=2j-1\).
  Here \(\ZZ(j)\) on the right-hand side
  denotes \((2\pi i)^j\ZZ\);
  its indicated generator identifies the target
  with \(H^{i}_{\Bet}(X;\ZZ)\).
\end{Theorem}

\subsection{Strategy}\label{ss:strategy}

For \(j\geq0\),
we construct in \cref{ss:motivic} a natural comparison morphism
\begin{equation*}
  b_{j,X}\colon
  \ZZ(j)^{\mot}(\Cls{C}(X))
  \to \Gamma_{\Bet}(X;\ZZ(j)).
\end{equation*}
We first establish the following rational and finite-coefficient
comparisons:

\begin{Theorem}\label{main-q}
  Let~\(X\) be a compactum and \(j\geq0\) an integer.
  The map induced by~\(b_{j,X}\)
  \begin{equation*}
    H^{i}_{\mot}(\Cls{C}(X);\QQ(j))
    \to H^{i}_{\Bet}(X;\QQ(j))
  \end{equation*}
  is bijective for \(i\geq2j\)
  and surjective for \(i=2j-1\).
\end{Theorem}

The finite-coefficient input is stronger:

\begin{Theorem}\label{main-p}
  Let~\(X\) be a compactum, \(j\geq0\) an integer,
  and \(p\) a prime.
  Reduction of~\(b_{j,X}\) modulo~\(p\) gives an equivalence
  \begin{equation*}
    \FF_{p}(j)^{\mot}(\Cls{C}(X))
    \simeq
    \Gamma_{\Bet}(X;\FF_{p}(j))
  \end{equation*}
  in~\(\D(\FF_p)\).
\end{Theorem}

\begin{proof}[Proof of \cref{main}]
  We write~\(F\)
  for the fiber of~\(b_{j,X}\).
  \Cref{main-p} shows that
  \(F\) is rational.
  \Cref{main-q} shows that
  \(\pi_{-i}F\simeq
  \pi_{-i}(F\otimes\QQ)\) vanishes
  for \(i\geq2j\).
\end{proof}

We prove \cref{main-q,main-p}
in \cref{s:proof}.
The bijectivity statements in both cases
are based
on the going-down argument of~\cite[Section~8]{k-ros-1},
which we review in \cref{ss:cd}.
The surjectivity statement in \cref{main-q}
follows from~\cite{mot-oka}.

\subsection*{Acknowledgments}\label{ss:ack}

I thank Peter Scholze for helpful discussions.
The results were largely obtained when I was
at the Max Planck Institute for Mathematics,
although the proof here is much simpler.
This work was supported by the Danish National Research Foundation
through the Copenhagen Center for Geometry and Topology (DNRF151).

\subsection*{Convention}\label{ss:conv}

We do not decorate derived functors with~\(\mathrm{L}\) or~\(\mathrm{R}\);
in particular, \(\Gamma\) denotes cohomology.

\section{Preparation}\label{s:preparation}

In \cref{ss:cd},
we recall the cd~topology of~\cite{k-ros-1}.
In \cref{ss:motivic},
we recall motivic cohomology.

\subsection{The cd topology}\label{ss:cd}

We recall the part of the cd~topology
from~\cite[Sections~3, 4, and~8]{k-ros-1}
that will be used below.
We write~\(\Cat{Cpt}_{\aleph_1}\)
for the category of compacta of countable weight.
A family~\(\{X_{i}\to X\}_{i}\) in~\(\Cat{Cpt}_{\aleph_1}\)
is a \emph{cd~cover}
if there is a filtration
by closed subsets
\begin{equation*}
  \emptyset=Z_0\subset Z_1\subset\dotsb\subset Z_n=X
\end{equation*}
such that, for every~\(1\leq k\leq n\), there is a map
\(q_k\colon Y_k\to Z_k\) such that the composite
\(Y_k\to Z_k\hookrightarrow X\) factors
through one of the maps \(X_i\to X\),
and \(q_k\) induces a homeomorphism
\begin{equation*}
  Y_k\setminus q_k^{-1}(Z_{k-1})
  \simeq
  Z_k\setminus Z_{k-1}.
\end{equation*}
The cd topology is the Grothendieck topology
generated by these covers;
see~\cite[Definitions~3.2 and~3.4]{k-ros-1}.
In particular,
every finite closed cover is a cd~cover.

Equivalently, by~\cite[Lemma~3.10]{k-ros-1},
the cd~topology is generated
by abstract blowup squares
of compacta.
More precisely, a square
\begin{equation*}
  \begin{tikzcd}
    Z'\ar[d]\ar[r]&
    X'\ar[d]\\
    Z\ar[r]&
    X
  \end{tikzcd}
\end{equation*}
in \(\Cat{Cpt}_{\aleph_1}\)
is an \emph{abstract blowup square} if it is cartesian,
\(Z\to X\) is a closed embedding, and
\(X'\setminus Z'\to X\setminus Z\) is a homeomorphism.
A presheaf on~\(\Cat{Cpt}_{\aleph_1}\)
is a cd~sheaf precisely
when it takes the empty set to a final object
and sends every such square to a cartesian square
by~\cite[Theorem~3.9]{k-ros-1}.

The following says that complex points carry éh~covers
to cd~covers
after pulling back to a compactum:

\begin{lemma}\label{covers}
  Let~\(A\) be a \(\QQ\)-algebra of finite type
  and \(X\to(\Spec A)(\CC)\) a continuous map
  from a compactum of countable weight.
  The pullback of every éh~covering family of~\(\Spec A\)
  admits a refinement
  by a cd~cover.
\end{lemma}

\begin{proof}
  By definition,
  the éh~topology is generated by the cdh and étale topologies,
  so it is enough to consider covers in either topology.
  For the cdh~topology,
  the assertion follows from~\cite[Theorem~4.4]{k-ros-1}.

  Now consider an étale covering family.
  An étale map is a local homeomorphism on complex points.
  The local sections of the given étale covering family
  therefore yield an open cover of~\(X\).
  By compactness and normality,
  it has a finite closed refinement,
  which is a cd~cover.
\end{proof}

We will use the following going-down lemma
from~\cite[Proposition~8.5]{k-ros-1}:

\begin{lemma}\label{going-down}
  Let \(E\to E'\) be a map of spectrum-valued cd~sheaves
  on \(\Cat{Cpt}_{\aleph_1}\).
  Suppose the following:
  \begin{conenum}
    \item\label{i:edge}
      For every~\(X\),
      the map
      \(\pi_{n}(E(X))\to\pi_{n}(E'(X))\) is an isomorphism.
    \item\label{i:triv}
      Every class in
      \(\pi_{m}(E(X))\) or~\(\pi_{m}(E'(X))\), for~\(m<n\),
      is killed by a cd~cover.
  \end{conenum}
  Then
  \begin{equation*}
    \tau_{\leq n}(E(X))\to\tau_{\leq n}(E'(X))
  \end{equation*}
  is an equivalence for every~\(X\).
\end{lemma}

\begin{example}\label{x557y5}
  Let~\(X\) be a compactum of countable weight.
  For~\(i>0\),
  every class in~\(H^i_{\Bet}(X;M)\)
  is killed by a finite closed cover of~\(X\),
  for any abelian group~\(M\).
  Indeed,
  the class is pulled back from a polyhedron,
  where a finite closed
  cover by contractible subsets kills it;
  see~\cite[Proposition~1.7]{k-ros-1}.
\end{example}

\subsection{Motivic cohomology}\label{ss:motivic}

For a \(\QQ\)-algebra~\(A\), let~\(\ZZ(j)^{\mot}(A)\)
denote the motivic cohomology of
Elmanto–Morrow~\cite{ElmantoMorrow}. Their construction glues the
\(\AA^1\)-motivic filtration on homotopy \(K\)-theory to the
Hochschild–Kostant–Rosenberg filtration on negative cyclic homology
along their cdh~sheafifications.
It comes with a natural comparison
\begin{equation*}
  \ZZ(j)^{\mot}(A)\to\ZZ(j)^{\AA}(A),
\end{equation*}
where the target is the \(\AA^1\)-invariant motivic cohomology of
Bachmann–Elmanto–Morrow~\cite{BachmannElmantoMorrow}; see
also~\cite[Sections~1.1–2]{ElmantoMorrow}.
In general,
for an abelian group~\(\Lambda\) and integers~\(i\) and~\(j\),
we write
\begin{align*}
  \Lambda(j)^{\X}&=\ZZ(j)^{\X}\otimes_{\ZZ}\Lambda,&
  H^i_{\X}(A;\Lambda(j))&=\pi_{-i}(\Lambda(j)^{\X}(A)).
\end{align*}

The \(K\)-regularity of~\(\Cls{C}(X)\)
is proven in~\cite[Theorem~1.5]{CortinasThom12};
see~\cite[Remark~8.2]{CortinasThom12}
for the earlier work of Rosenberg and Friedlander–Walker.
The corresponding statement for motivic cohomology follows from this;
see~\cite[Theorem~6.1]{Bouis26}:

\begin{theorem}[Bouis]\label{bouis}
  For every compactum~\(X\) and every~\(j\geq0\),
  the map
  \(\ZZ(j)^{\mot}(\Cls{C}(X))\to\ZZ(j)^{\AA}(\Cls{C}(X))\)
  is an equivalence in~\(\D(\ZZ)\).
\end{theorem}

Therefore,
in this paper, we can work with \(\AA^1\)-invariant motivic cohomology.

We recall the properties of \(\ZZ(j)^{\AA}\).
It is a finitary cdh~sheaf
by~\cite[Theorem~1.1\,(3)]{BachmannElmantoMorrow}.
It is represented by an absolute motivic spectrum
by~\cite[Theorem~1.6]{BachmannElmantoMorrow},
so it satisfies Milnor
excision by~\cite[Corollary~1.2]{EHIK21M}.
If a prime~\(p\) is invertible in~\(A\),
we have the Beilinson–Lichtenbaum equivalence of cdh~sheaves
\begin{equation}\label{e:bl}
  \FF_p(j)^{\AA}
  \simeq
  L_{\cdh}\tau_{\geq-j}
  \Gamma_{\et}(\X;\mu_p^{\otimes j}),
\end{equation}
by~\cite[Theorem~1.1\,(4)]{BachmannElmantoMorrow}.
After rationalization,
the motivic filtration splits naturally into
Adams weights:
\begin{equation}\label{e:splitting}
  \fil_{\mot}^{*}{\KH(A)\otimes\QQ}
  \simeq
  \bigoplus_{j\geq *}
  \QQ(j)^{\AA}(A)[2j].
\end{equation}
Here, for every integer~\(a\geq1\),
the Adams operation~\(\psi^a\) acts
on the weight-\(j\) summand by~\(a^j\);
see~\cite[Theorems~4.47 and~4.48]{BachmannElmantoMorrow}.

We finish by constructing the comparison map~\(b_{j,X}\).
Let~\(\ZZ(j)^{\lis}\) denote the left Kan extension of classical
motivic cohomology from smooth \(\QQ\)-algebras.
On \(\QQ\)-algebras of finite type,
\(\ZZ(j)^{\AA}\) is
the cdh~sheafification of~\(\ZZ(j)^{\lis}\)
by~\cite[Theorem~7.5, Remark~7.9, and
Corollary~1.11]{BachmannElmantoMorrow}.
Classical Betti realization gives a natural transformation
\begin{equation*}
  \ZZ(j)^{\lis}(\X)
  \to
  \Gamma_{\Bet}((\Spec\X)(\CC);\ZZ(j))
\end{equation*}
for \(\QQ\)-algebras of finite type.
The target satisfies cdh~descent
by~\cite[Theorems~4.4 and~3.11]{k-ros-1}.
Consequently,
Betti realization factors through cdh~sheafification.
For every \(\QQ\)-algebra~\(A\) of finite type, it defines
\begin{equation}
  \ZZ(j)^{\AA}(A)
  \to
  \Gamma_{\Bet}((\Spec A)(\CC);\ZZ(j)).
\end{equation}

\begin{definition}\label{xou9xb}
  Let \(X\) be a compactum
  and \(j\geq0\) an integer.
  For a finitely generated \(\QQ\)-subalgebra~\(A\subset\Cls{C}(X)\),
  we consider
  the composite
  \(\ZZ(j)^{\AA}(A)\to\Gamma_{\Bet}((\Spec A)(\CC);\ZZ(j))\to
  \Gamma_{\Bet}(X;\ZZ(j))\),
  where the first map is the one above.
  By considering the filtered colimit over all such~\(A\),
  we obtain the comparison map:
  \begin{equation*}
    b^{\AA}_{j,X}\colon
    \ZZ(j)^{\AA}(\Cls{C}(X))
    \to \Gamma_{\Bet}(X;\ZZ(j)).
  \end{equation*}
\end{definition}

\section{Proof}\label{s:proof}

We prove \cref{main-q,main-p}
in \cref{ss:rational,ss:finite},
respectively.
The main tool is \cref{going-down}.
We first note the following:

\begin{proposition}\label{x0uwwc}
  Let~\(j\geq0\) be an integer
  and \(\Lambda\) an abelian group.
  The functors
  \(\Cat{Cpt}^{\op}\to\D(\ZZ)\) given by
  \begin{align*}
    X&\mapsto\Lambda(j)^{\AA}(\Cls{C}(X)),&
    X&\mapsto\Gamma_{\Bet}(X;\Lambda(j))
  \end{align*}
  preserve \(\aleph_{1}\)-filtered colimits
  and are cd~sheaves when restricted to~\(\Cat{Cpt}_{\aleph_1}\).
\end{proposition}

\begin{proof}
  The functors \(\Cls{C}(\X)\),
  \(\ZZ(j)^{\AA}\),
  and \(\Gamma_{\Bet}(\X;\ZZ(j))\)
  preserve \(\aleph_{1}\)-filtered colimits.

  For cd~descent,
  we use the description by abstract blowup squares.
  The functor~\(\Cls{C}(\X)\) carries abstract blowup squares
  of compacta to Milnor squares of rings.
  Thus the first functor satisfies cd~desecnt
  by~\cite[Corollary~1.2]{EHIK21M},
  and the second by~\cite[Theorem~3.11]{k-ros-1}.
\end{proof}

\subsection{Rational coefficients}\label{ss:rational}

\begin{proof}[Proof of \cref{main-q}]
  For the isomorphism statement, by \cref{bouis,x0uwwc},
  it is enough to consider~\(b^{\AA}_{j,\X}\)
  on \(\Cat{Cpt}_{\aleph_{1}}^{\op}\).
  By \cref{x0uwwc}, we may apply \cref{going-down} to
  \begin{equation*}
    b^{\AA}_{j,\X}\otimes\QQ\colon
    \QQ(j)^{\AA}(\Cls{C}(\X))
    \to \Gamma_{\Bet}(\X;\QQ(j))
  \end{equation*}
  with~\(n=-2j\).
  The conditions~\cref{i:edge,i:triv}
  are verified in \cref{rational-edge,rational-triv} below.
  This gives the required isomorphism.

  For the surjectivity statement,
  by~\cite[Proposition~3.1]{mot-oka},
  lisse and Lichtenbaum motivic cohomology agree rationally.
  Under this identification, by construction,
  the composite
  \begin{equation*}
    H^{2j-1}_{\lis}(\Cls{C}(X);\QQ(j))
    \to
    H^{2j-1}_{\AA}(\Cls{C}(X);\QQ(j))
    \to
    H^{2j-1}_{\Bet}(X;\QQ(j))
  \end{equation*}
  is the rationalization of
  the cycle map in~\cite[Definition~2.10]{mot-oka}.
  Its surjectivity follows from~\cite[Theorem~C]{mot-oka}.
\end{proof}

\begin{lemma}\label{rational-edge}
  Let~\(X\) be a compactum and \(j\geq0\) an integer.
  Then~\(b^{\AA}_{j,X}\) induces an isomorphism
  \begin{equation*}
    H^{2j}_{\AA}(\Cls{C}(X);\QQ(j))
    \simeq
    H^{2j}_{\Bet}(X;\QQ(j)).
  \end{equation*}
\end{lemma}

\begin{proof}
  The rational splitting~\cref{e:splitting}
  identifies
  \(
    H^{2j}_{\AA}(\Cls{C}(X);\QQ(j))
  \)
  with the Adams weight-\(j\) summand
  of~\(\KH_0(\Cls{C}(X))\otimes\QQ\).
  By~\cref{xhar7u}
  and \(K\)-regularity,
  this is the weight-\(j\) summand
  of~\(\ku^0(X)\otimes\QQ\);
  here we use that
  the Serre–Swan comparison commutes
  with exterior powers
  and hence with Adams operations.

  We spell out compatibility
  with the actual map~\(b^{\AA}_{j,X}\).
  Every class in~\(K_0(\Cls{C}(X))\)
  is a difference of classes represented by idempotent matrices.
  On each of the finitely many clopen rank strata,
  such a matrix is classified
  by the smooth affine \(\QQ\)-scheme
  of idempotents of fixed rank.
  Thus the class is pulled back
  from a universal vector bundle
  over a smooth \(\QQ\)-algebra.
  There, the rational splitting agrees
  with the usual Adams decomposition
  by~\cite[Theorem~4.47\,(4)]{BachmannElmantoMorrow},
  while Betti realization carries the algebraic Chern classes
  of this bundle to the corresponding topological Chern classes
  by, e.g.,~\cite[Theorems~27 and~28]{FriedlanderWalker05}.
  Pullback and additivity show
  that~\(\pi_{-2j}(b^{\AA}_{j,X})\)
  is the weight-\(j\) component
  of the Serre–Swan comparison.
  The rational topological Chern character identifies this component
  with~\(H^{2j}_{\Bet}(X;\QQ(j))\)
  by, e.g.,~\cite[Theorems~V.3.25 and~V.3.27]{Karoubi78}.
\end{proof}

\begin{lemma}\label{rational-triv}
  Let~\(X\) be a compactum of countable weight
  and
  \(i\) and~\(j\) integers
  satisfying \(i>2j\geq0\).
  Every class in either
  \(H^i_{\AA}(\Cls{C}(X);\QQ(j))\)
  or \(H^i_{\Bet}(X;\QQ(j))\)
  is killed by a cd~cover.
\end{lemma}

\begin{proof}
  Any Betti class is killed by a cd~cover by~\cref{x557y5}.

  Let~\(\alpha\in H^i_{\AA}(\Cls{C}(X);\QQ(j))\).
  By finitarity, the class is represented
  on a \(\QQ\)-algebra~\(A\) of finite type mapping to~\(\Cls{C}(X)\).
  Resolution of singularities gives a cdh~cover
  of~\(\Spec A\) by a smooth affine scheme~\(\Spec B\).
  There,
  \(\ZZ(j)^{\AA}(B)\) is described via Bloch’s cycle complex
  by~\cite[Theorem~1.1\,(8)]{BachmannElmantoMorrow},
  and hence
  \begin{equation*}
    H^i_{\AA}(B;\QQ(j))
    \simeq \CH^j(B,2j-i)\otimes\QQ=0
  \end{equation*}
  since~\(i>2j\).
  By~\cref{covers},
  the pullback of this cover to~\(X\)
  admits a refinement by a cd~cover,
  which kills the original class.
\end{proof}

\subsection{Finite coefficients}\label{ss:finite}

\begin{proof}[Proof of \cref{main-p}]
  By \cref{bouis,x0uwwc}, it is enough to consider
  \(b^{\AA}_{j,\X}\) on \(\Cat{Cpt}_{\aleph_{1}}^{\op}\).
  By \cref{x0uwwc}, we may apply \cref{going-down}
  to
  \begin{equation*}
    b^{\AA}_{j,\X}\otimes\SS/p\colon
    \FF_{p}(j)^{\AA}(\Cls{C}(\X))
    \to\Gamma_{\Bet}(\X;\FF_{p}(j))
  \end{equation*}
  with~\(n=0\).
  The conditions~\cref{i:edge,i:triv}
  are verified in \cref{finite-edge,finite-triv} below.
  The going-down criterion therefore
  gives an isomorphism on every nonpositive homotopy group.
  Since the source is \(0\)-truncated by~\cref{e:bl},
  the comparison is an equivalence.
\end{proof}

\begin{lemma}\label{finite-edge}
  Let~\(X\) be a compactum, \(j\geq0\) an integer,
  and \(p\) a prime.
  Then~\(b^{\AA}_{j,X}\) induces an isomorphism
  \begin{equation*}
    H^0_{\AA}(\Cls{C}(X);\FF_p(j))
    \simeq H^0_{\Bet}(X;\FF_p(j)).
  \end{equation*}
\end{lemma}

\begin{proof}
  The degree-zero case of~\cref{e:bl} gives a natural isomorphism
  \begin{equation*}
    H^0_{\AA}(\Cls{C}(X);\FF_p(j))
    \simeq
    H^0_{\et}(\Spec\Cls{C}(X);\mu_p^{\otimes j}).
  \end{equation*}
  The right-hand side consists of the locally constant
  \(\mu_p^{\otimes j}\)-valued functions on~\(X\):
  Clopen subsets of~\(X\) correspond to idempotents
  in~\(\Cls{C}(X)\),
  which are \(\{0,1\}\)-valued continuous functions.
\end{proof}

\begin{lemma}\label{finite-triv}
  Let~\(X\) be a compactum of countable weight,
  \(j\geq0\) an integer, and \(p\) a prime.
  For \(i>0\),
  every class in either
  \(H^i_{\AA}(\Cls{C}(X);\FF_p(j))\)
  or \(H^i_{\Bet}(X;\FF_p(j))\)
  is killed by a cd~cover.
\end{lemma}

\begin{proof}
  Any Betti class is killed by a cd~cover by~\cref{x557y5}.

  Let~\(\alpha\in H^i_{\AA}(\Cls{C}(X);\FF_p(j))\).
  By finitarity,
  it comes from a class
  in \(H^{i}_{\AA}(A;\FF_{p}(j))\)
  where \(A\) is a \(\QQ\)-algebra
  of finite type mapping to~\(\Cls{C}(X)\).

  When~\(i\leq j\),
  \cref{e:bl} identifies
  the class on~\(A\) with an étale cohomology class
  with coefficients in~\(\mu_p^{\otimes j}\).
  Since~\(i>0\), this class vanishes on an étale covering family.
  By~\cref{covers},
  the pullback of this covering family to~\(X\)
  admits a refinement by a cd~cover,
  which kills the original class.

  When~\(i>j\),
  \cref{e:bl} shows that~\(\alpha\)
  vanishes on a cdh~cover of~\(\Spec A\).
  By~\cref{covers},
  the pullback of this cover to~\(X\)
  admits a refinement by a cd~cover
  which kills the original class.
\end{proof}

\bibliographystyle{alpha}
\let\AA\oldAA \let\SS\oldSS  \newcommand{\yyyy}[1]{}

\end{document}